\documentclass[reqno]{amsart}
\usepackage{amssymb,amscd,verbatim, amsthm,graphicx, color}
\usepackage{fancybox}
\usepackage{longtable}

\newcommand \C[1]{{\mathcal #1}}

\newcommand{\Wp}{{W'}}
\newcommand{\ev}{{\mathsf{ev}}}
\newcommand{\elli}{{\mathsf{ell}}}
\newcommand{\temp}{{\mathsf{temp}}}
\newcommand{\res}{{\mathsf{res}}}
\newcommand{\ext}{{\mathsf{ex}}}
\newcommand{\aff}{{\mathsf{aff}}}
\newcommand{\lra}{{\longrightarrow}}
\newcommand{\sfi}{{i}}
\newcommand{\ol}{\overline}

\newcommand \wti[1]{{\widetilde {#1}}}

\newcommand\frg{\mathfrak g}
\newcommand\frh{\mathfrak h}

\newcommand{\Ahat}{\widehat{A}}

\newcommand \bC{{\mathbb C}}
\newcommand \bbC{{\mathbb C}}

\newcommand \bH{{\mathbb H}}
\newcommand \bbH{{\mathbb H}}
\newcommand \bR{{\mathbb R}}
\newcommand \bZ{{\mathbb Z}}

\newcommand\CH{{\C H}}
\newcommand\caH{{\C H}}

\newcommand\CX{{\C X}}
\newcommand\CY{{\C Y}}

\newcommand\caR{{\C R}}
\newcommand\caD{{\C D}}

\newcommand{\EP}{\mathrm{EP}}

\newcommand\ep{{\epsilon}}

\newcommand\om{{\omega}}
\newcommand\al{{\alpha}}

\newtheorem{theorem}{Theorem}[section]

\newtheorem{corollary}[theorem]{Corollary}
\newtheorem{lemma}[theorem]{Lemma}
\newtheorem{proposition}[theorem]{Proposition}

\theoremstyle{definition}
\newtheorem{remark-plain}{Remark}[section]
\newtheorem*{example-plain}{Example}

\newcommand\Hom{\operatorname{Hom}}

\newcommand\End{\operatorname{End}}

\newcommand\im{\operatorname{im}}

\newcommand\triv{\mathsf{triv}}
\newcommand\sgn{\mathsf{sgn}}

\numberwithin{equation}{section}

\begin{document}

\title[characters of springer representations] {characters of springer representations on elliptic
conjugacy classes}

\author{Dan Ciubotaru}
        \address[D. Ciubotaru]{Dept. of Mathematics\\ University of
          Utah\\ Salt Lake City, UT 84112}
        \email{ciubo@math.utah.edu}

\author{Peter E.~Trapa}
        \address[P. Trapa]{Dept. of Mathematics\\ University of
          Utah\\ Salt Lake City, UT 84112}
        \email{ptrapa@math.utah.edu}

\thanks{D.C.~is partially supported by NSF-DMS 0968065 and NSA-AMS
  081022. P.T.~is partially supported by NSF-DMS 0968275. The authors
  thank M. Solleveld for helpful comments.}

\begin{abstract}
For a Weyl group $W$, we give a simple closed formula (valid on
elliptic conjugacy classes) for the character of the representation of
$W$ in each $A$-isotypic component of the full homology of a Springer
fiber.  We also give a formula (valid again on elliptic conjugacy
classes) of the $W$-character of an irreducible discrete series
representation with real central character of a graded affine Hecke
algebra with arbitrary parameters.  In both cases, the Pin double
cover of $W$ and the Dirac operator for graded affine Hecke algebras
play key roles.
\end{abstract}

\maketitle

\section{introduction}\label{s:intro}
Let $W$ denote a Weyl group acting by the reflection representation in a real vector space
$V_0$, and recall than an element of $W$ is called elliptic if it has no
fixed points on $V_0$.  The main result of this paper gives a simple
closed formula, valid on the set of elliptic elements in $W$, for
 the character of the $W$ representation on each $A$-isotypic
component of the full homology
of a Springer fiber.  
The formula depends on two ingredients --- the computation of the
Springer correspondence (in the top degree of homology) and
the character table of a certain double cover of $W$ --- both of which
are known.  In particular, our approach is independent of the 
Lusztig-Shoji algorithm.

In more detail, let $\Phi = (R,X,R^\vee,Y)$ 
be a crystallographic root system as in Section \ref{ss:rs}, and let $W$ denote the corresponding
Weyl group.  Let $e$ denote a nilpotent element in the complex
semisimple Lie algebra $\frg$ attached to the root system $\Phi$, and
let $A(e)$ denote the component group of centralizer in
$\mathrm{Ad}(\frg)$ of $e$.  Then Springer has defined an action of $W
\times A(e)$ on the homology $H_\bullet(\C B_e,\bZ)$ where $\C B_e$
denotes the Springer fiber over $e$.  For a fixed irreducible
representation $\phi$ of $A(e)$, write $\chi_{e,\phi}$ for the
character of the $W$ representation on the $\phi$-isotypic component
of $H_\bullet(\C B_e,\bZ)$, and assume this space is nonzero.  Write
$\sigma(e,\phi)$ for the irreducible representation of $W$ in the top degree
of homology.

Assuming $\chi_{e,\phi}$ is not identically zero on the elliptic set,
  Theorem \ref{t:intro} gives a formula for the value of
  $\chi_{e,\phi}$ on any elliptic element of $W$.
  On the other hand, Reeder \cite{R} has shown that $\chi_{e,\phi}$
  vanishes on the set of elliptic elements in $W$
  precisely when $e$ fails to be a quasidistinguished nilpotent
  element in the sense of \cite[(3.2.2)]{R}. 
        Taken together,
  this completely determines the values on the elliptic set of any
  character of the form $\chi_{e,\phi}$.

To state our formula, we first fix a choice of positive roots $R^+$ and
a $W$-invariant inner product $\langle~,~\rangle$ on $V_0 = X
\otimes_\bZ \bR$.  The group $\mathsf{Pin}(V_0)$ is a subgroup of
units in the Clifford algebra $C(V_0)$, and maps surjectively onto the
orthogonal group $\mathsf{O}(V_0)= \mathsf{O}(V_0,\langle~,~\rangle)$
with kernel of order two.  Since $W$ acts by orthogonal transformation
on $V_0$, we can consider its preimage $\wti W$ in $\mathsf{Pin}(V_0)$.
Write $p$ for the projection from $\wti W$ to $W$.

For the purposes of the introduction, assume $\dim(V_0)$ is odd.  (The
even case is virtually identical, but introduces some notation we
prefer to avoid here.)  Then, up to isomorphism, $C(V_0)$ has exactly
two inequivalent complex simple modules which we denote $S^+$ and
$S^-$.  Both remain irreducible when restricted to $\wti W$ and are
genuine in the sense that they do not factor to $W$.  Write
$\chi_{S^+}$ and $\chi_{S^-}$ for their characters.  Propositions
\ref{p:s-squared} and \ref{p:wedgepm} below show that (up to sign)
\begin{equation}
\label{e:spm-intro}
\chi_{S^+}(\tilde w) -\chi_{S^-}(\tilde w) = \det(1-p(\tilde w))^{1/2},
\end{equation}
and hence is nonzero if and only if $p(\tilde w)$ is elliptic.
The final ingredient we need
is the Casimir element
\[
\Omega_{\wti W} 
= \frac14 \sum_{\substack{\al,\beta>0\\s_\al(\beta)<0}}
\frac{|\alpha^\vee|}{|\alpha|} \frac{|\beta^\vee|}{|\beta|} \alpha \beta \in C(V_0).
\]
In fact $\Omega_{\wti W}$ 
is a central element in $\bC[\wti W]$,
and hence acts by a scalar in any irreducible representation of $\wti W$.
Our main result, proved at the end of
Section \ref{ss:equal} below, is as follows.

\begin{theorem}
\label{t:intro}
Suppose $\chi_{e,\phi}$ is not identically zero on the set of elliptic
elements of $W$.   Fix
$w \in W$ elliptic and choose $\tilde w$ in $\wti W$ such that
$p(\tilde w) = w$.  Then there exist two distinct  irreducible genuine
representation $\tilde \sigma^+$ and $\tilde \sigma^- = \tilde
\sigma^+ \otimes \sgn$ of $\wti W$ such that
\begin{equation}
\label{e:intro}
\chi_{e,\phi}(w) = \frac{\chi_{\tilde \sigma^+}(\tilde w)-\chi_{\tilde \sigma^-}(\tilde w)}
{\chi_{S^+}(\tilde w) -\chi_{S^-}(\tilde w)}
\end{equation}
where $\chi_{\tilde \sigma^{\pm}}$ denotes the 
character of $\tilde \sigma^\pm$.  Up to tensoring with $\sgn$, $\tilde \sigma^+$ is
characterized as the unique (multiplicity one) constituent of
\[
\sigma(e,\phi) \otimes S^+
\]
for which the scalar by which $\Omega_{\wti W}$ acts is as small as possible.
\end{theorem}
\noindent
Note that the denominator in \eqref{e:intro} is nonzero by
\eqref{e:spm-intro}, and also that the quotient in \eqref{e:intro}
is independent of the choice of $\tilde w$ since both $S^\pm$ and
$\tilde \sigma^\pm$ are genuine representations of $\wti W$.  

\medskip

According to the classification of Kazhdan-Lusztig \cite{KL} and 
Lusztig \cite{L,l:cls2}, the
characters $\chi_{e,\phi}$ are the restrictions to $W$ of the
irreducible tempered modules with real central character for a
(graded) affine Hecke algebra $\bH$ with equal parameters.  So Theorem
\ref{t:intro} may be interpreted as a result about the tempered $\bH$ modules
whose $W$ characters do not vanish on the elliptic set.
With this in mind, the idea of the proof of Theorem \ref{t:intro} is as follows.
Fix such a tempered $\bH$ module $X$.  The key object of study for us is
\[
I(X) := X \otimes (S^+ - S^-)
\]
introduced in Section \ref{ss:index}.
On one hand, in Lemma \ref{l:relation},
we related $I(X)$ to the elliptic representation theory developed by Schneider-Stuhler
\cite{SS}, Reeder \cite{R}, and Opdam-Solleveld \cite{OS}.  From this
we deduce that $I(X)$ has the simple form $\tilde \sigma^+ - \tilde \sigma^-$,
and dividing by $\chi_{S^+} - \chi_{S^+}$ leads to \eqref{e:intro}.  On the other
hand, in Proposition \ref{p:index}, we relate $I(X)$ to the index of the Dirac
operator defined in \cite{BCT}.  A formula for the square of the Dirac operator
then imposes strict limitations on the possibilities for $\tilde \sigma^\pm$. 
Together with \cite{C}, they lead to the explicit description of $\tilde \sigma^+$
given in the theorem.

In the setting of arbitrary parameters, 
our argument leads to the following result.  It bears a strong formal
resemblance to Harish-Chandra's character formulas for the discrete series
of a semisimple Lie group.

\begin{theorem}
\label{t:intro2}
Let $\bH$ be a graded affine Hecke algebra attached to the root system
$\Phi$ and arbitrary real positive parameters as defined in Section
\ref{ss:daha}.  Let $X$ be an irreducible discrete series module for
$\bH$, and write $\chi_{X}$ for the character of the $W$
representation afforded by $X$.  
Fix $w \in W$ elliptic
and choose $\tilde w$ such that $p(\tilde w) = w$.  Then there exist
two inequivalent genuine irreducible representation $\tilde \sigma^+$
and $\tilde \sigma^- = \tilde \sigma^+ \otimes \sgn$ of $\wti W$ such
that
\begin{equation}
\label{e:intro2}
\chi_{X}(w) = \frac{\chi_{\tilde \sigma^+}(\tilde w)-\chi_{\tilde \sigma^-}(\tilde w)}
{\chi_{S^+}(\tilde w) -\chi_{S^-}(\tilde w)}
\end{equation}
\end{theorem}

\medskip
Like the case of equal parameters (Corollary \ref{c:diffeq}), we expect that this result holds for any
irreducible tempered module with real central character
whose $W$ character does not vanish on the elliptic set, and we also
expect that the crystallographic condition can be dropped.
We remark that when the parameters
are of geometric origin (as computed in \cite{l:cls2}), then the results
of \cite{C} again apply to give an explicit description of  $\tilde
\sigma^\pm$ by a procedure similar to the
one given in Theorem \ref{t:intro}.

\begin{example-plain}
\label{ex:intro}
Let $\Phi$ be the root system of type $B_n$, and let $\bH_{n,m}$,
$m>0,$ be the graded affine Hecke algebra for $\Phi$ with parameters
 $c(\al_l)=1,$ for a long
root $\al_l$, and  $c(\al_s)=m$, for a short root $\al_s.$ Fix a partition $\sigma$ of
$n$. We attach to $\sigma$ and $m$ a real central character
$c_m(\sigma)$ (see the procedure described at the beginning of Section
3 in \cite{CKK}, for example). Opdam \cite[Lemma 3.31]{O} showed that when $m\notin \frac
12\bZ$, there exists a unique discrete series $\bH_{n,m}$-module
$\mathsf{ds}_m(\sigma)$ with central character $c_m(\sigma).$ 

Theorem
\ref{t:intro2} gives a simple formula for the $W(B_n)$-character of
$\mathsf{ds}_m(\sigma)$ on the set of elliptic elements.
Recall that the conjugacy classes of elliptic elements in
$W(B_n)$ are in one-to-one correspondence with partitions $\lambda$ of
$n$. More precisely, to every partition $\lambda=(n_1\ge\dots\ge
n_\ell>0)$, one attaches the conjugacy class of  $w_\lambda$, a Coxeter element in
$W(B_{n_1})\times\dots\times W(B_{n_\ell})\subset W(B_n)$. 
Since $W(B_n)=S_n\ltimes (\bZ/2\bZ)^n,$ every $W(B_n)$-type is induced from a character $\chi$ of $(\bZ/2\bZ)^n$ tensored with an irreducible representation of the stabilizer in $S_n$ of $\chi$. Let $(\sigma\times\emptyset)$ denote the $W(B_n)$-type obtained from the trivial character of $(\bZ/2\bZ)^n$ and the $S_n$-type parameterized by $\sigma.$
By \cite[Sections 3.7 and 3.10]{C}, the representation $\wti\sigma^+$ occuring in (\ref{e:intro2}) when $X= \mathsf{ds}_m(\sigma)$ equals  $(\sigma\times\emptyset)\otimes S^+$ (up to tensoring with $\sgn$). 
Then, up to a sign depending only
on $\mathsf{ds}_m(\sigma)$, Theorem \ref{t:intro2} reduces to
\begin{equation*}
\chi_{\mathsf{ds}_m(\sigma)}(w_\lambda)=\chi_{\sigma\times\emptyset}(w_\lambda)=\chi_\sigma^{S_n}((\lambda));
\end{equation*}  
here $\chi_\sigma^{S_n}$ denotes the character of the $S_n$-representation
labelled by the partition $\sigma$, and $(\lambda)$ denotes the
$S_n$-conjugacy class with cycle structure $\lambda$. In particular,
this shows that, up to a sign, the character of
$\mathsf{ds}_m(\sigma)$ on the elliptic set is independent of $m$. The
same result also follows from \cite[Algorithm 3.30]{CKK}.

It is natural to ask if the same ``independence of parameter'' holds for the characters (on the elliptic set) of families of discrete series (in the sense of \cite[Section 3]{O}) for Hecke algebras with unequal parameters attached  to other multiply-laced root systems. Using (\ref{e:intro2}) and the explicit characters  $\wti\sigma^\pm$ given by \cite[Tables 1,2,6,7]{C}, one can easily verify that this is the case for the graded affine Hecke algebras of types $G_2$ and $F_4$ with geometric parameters.
\end{example-plain}

\section{preliminaries}
\label{s:prelim}

\subsection{Root data} 
\label{ss:rs}
Fix a (reduced) root datum
$\Phi=(R\subset X, R^\vee \subset Y)$.  Thus $X$ and $Y$ are
finite-rank lattices, there exists a perfect
$\bZ$-linear pairing denoted $(~,~)$ from
$X\times Y$ to $\bZ$, and a bijection $\alpha \mapsto \alpha^\vee$
from $R$ to $R^\vee$ such that
\[
(\alpha,\alpha^\vee) =2
\]
and so that
\begin{align*}
s_\al(v)&:=v-(v,\al^\vee)\al \\
s_{\al^\vee}(v')&:=v'-(\alpha,v')\al^\vee
\end{align*}
preserve $R$ and $R^\vee$ respectively.  We further assume $R$ 
spans $V_0:=X\otimes_\bZ \bR$.

Let $W$ denote the subgroup of $GL(X)$ generated by $\{s_\al \; | \;
\al \in R\}$. 
Let $\sgn$ denote the character of $W$ obtained by composing the inclusion
$W \subset GL(X)$ with the determinant.  Set $W_\ev = \ker(\sgn)$, the even
subgroup of $W$.  Because of a well-known dichotomy which appears below for
simple modules of the Clifford algebra, we find it convenient to define
\begin{equation}
\label{e:wp}
\Wp = 
\begin{cases}
W & \text{ if $\dim(V_0)$ is odd;}\\
W_\ev & \text{ if $\dim(V_0)$ is even.}
\end{cases}
\end{equation}

Choose a system of positive roots $R^+ \subset R$ and let $\Pi$ denote
the corresponding simple roots in $R^+$.  As usual, write $\al>0$ or
$\al<0$ in place of $\al\in R^+$ or $\al\in (-R^+)$, respectively.

Finally we fix a $W$-invariant inner product $\langle \cdot, \cdot
\rangle$ on $V_0$ and use the same notation for its extension to an
inner product on $V = V_0\otimes_\bR \bC$.  

\subsection{The Clifford algebra and the Pin double cover $\wti W$.}
\label{ss:cliff}
Let $ C(V_0)$ denote the Clifford algebra defined by $V_0$
and $\langle~,~\rangle$. 
Thus  {$C(V_0)$} is the
quotient of the tensor algebra of {$V_0$} by the ideal generated by
$$\om\otimes \om'+\om'\otimes \om+2\langle \om,\om'\rangle,\quad
\om,\om'\in {V_0}.$$
Let $\mathsf{O}{(V_0)}$ 
denote the group of orthogonal transformation of
{$V_0$} with respect to $\langle~,~\rangle$. The action of $-1\in
\mathsf{O}{(V_0)}$ induces a grading
\begin{equation*}
{C(V_0)= C(V_0)_{\ev}+  C(V_0)_{\mathsf{odd}}}.
\end{equation*}
Let $\ep$ be the automorphism of ${C(V_0)}$ which is $+1$ on $
{C(V_0)_{\ev}}$ and $-1$ on ${C(V_0)_{\mathsf{odd}}}$.
Let ${}^t$ be the transpose antiautomorphism of {$C(V_0)$}
characterized by
\begin{equation*}
\om^t=-\om,\ \om\in V_0,\quad (ab)^t=b^ta^t,\ a,b\in {C(V_0)}.
\end{equation*}
The Pin group is defined as
\begin{equation*}\label{pin}
{\mathsf{Pin}(V_0)}=\{a\in  {C(V_0)}\; |\; \ep(a) {V_0} 
a^{-1}\subset
{{V_0},~ a^t=a^{-1}\}}.
\end{equation*}
and we have an exact sequence
\begin{equation*}\label{ses}
1\longrightarrow \{\pm 1\} \longrightarrow
\mathsf{Pin}(V_0)\xrightarrow{\ \ p\ \ } \mathsf{O}(V_0)\longrightarrow 1,
\end{equation*}
where the projection $p$ is given by $p(a)(\om)=\ep(a)\om a^{-1}$.
The appearance of $\ep$ in the definition of $\mathsf{Pin}(V_0)$
insures that $p$ is surjective.

Since $W$ acts by orthogonal
transformations of $V_0$, we may define
\begin{equation}
\label{e:wtilde}
\wti W = p^{-1}(W)
\end{equation}
and
\begin{equation}
\label{e:wptilde}
\wti \Wp = p^{-1}(\Wp)
\end{equation}
where $W'$ is defined in \eqref{e:wp}.

\subsection{Notation for representation rings}
\label{ss:repring}
Given a ring $R$, we let $\caR(R)$ denote the integral Grothendieck group of finite-length left $R$-modules.  In the special case that $R=\bC[G]$, the group algebra
of a finite group $G$, we write $\caR(G)$ instead of $\caR(\bC[G])$.  
As usual, define a bilinear pairing on $\caR(G)$ via
\begin{equation*}
\label{e:bilinear-pairing}
\left(\sigma_1, \sigma_2 \right )_{G} 
= \dim \Hom_{G}(\sigma_1, \sigma_2)
\end{equation*}
for simple $\bC[G]$ modules $\sigma_i$ (and 
extended linearly to $\caR(G)$).  In terms of the characters $\chi_i$ of $\sigma_i$,
\[
\left(\sigma_1, \sigma_2 \right )_{G} = \frac{1}{|G|}\sum_{g\in G} \ol{\chi_1(g)}\chi_2(g).
\]

\subsection{Simple $C(V_0)$ modules and their restrictions to $\wti \Wp$.}
\label{ss:spinor}
The point of this subsection is to define two inequivalent
irreducible representations $(\gamma^\pm, S^\pm)$ of the group
$\wti \Wp$ defined in \eqref{e:wptilde}.

Suppose first that $\dim(V_0) =2r$ is even (so in particular $\wti \Wp$ is a
subgroup of $C(V_0)_\ev$).  Then, up to equivalence,
$C(V_0)$ has a unique simple (complex) module $S$ which remains
irreducible when restricted to $\wti W$.  Its restrictions to
$C(V_0)_\ev$, however, splits into two inequivalent modules, each of
dimension $2^{r-1}$, and each of which restrict irreducibly to the
group $\wti \Wp$.  We denote these two $\wti
\Wp$ representations by $(\gamma^+,S^+)$ and $(\gamma^-, S^-)$.  
By construction the action of $\wti W/ \wti \Wp$ interchanges $S^+$ and $S^-$.
The
notational superscripts $\pm$ are arbitrary, and a choice must be made when
fixing them.

Next suppose $\dim(V_0^\vee) =2r+1$ is odd.  Then, up to equivalence,
there is unique complex simple module $S$ for $C(V_0)_\ev$ which may
be extended naturally in two distinct ways to obtained inequivalent
simple module structures $S^+$ and $S^-$ on the same space $S$.  (The
choice of superscript labels is again arbitrary.)  Each
of these modules remains irreducible when restricted to $\wti W = \wti
\Wp$, and we continue to denote these two $\wti \Wp$ representations by
$(\gamma^+,S^+)$ and $(\gamma^-, S^-)$.  By construction, the action of
$\wti w \in \wti W$ in $S^+$ and $S^-$ differ by tensoring with
$\sgn$,
\begin{equation}
\label{e:signtwist}
\gamma^+(\wti w)s = \sgn(p(\wti w)) \gamma^-(\wti w)s.
\end{equation}

\begin{proposition}
\label{p:s-squared}
Recall the construction of $S^\pm$ given above, and the
notation of \eqref{e:wptilde} and Section \ref{ss:repring}.
Set
\begin{equation}
\label{e:wedgepm}
\wedge^\pm V = \sum_i (-1)^i \wedge^iV,
\end{equation}
regarded as an element of $\caR(\wti \Wp)$.
Then
\begin{equation}
\label{e:s-squared}
\left ( S^+ - S^-\right )^{\otimes 2} = \frac{2}{[W:\Wp]} \wedge^\pm V
\end{equation}
as elements of $\caR(\wti \Wp)$.  (Here $[W:\Wp]$ denotes
the index of $\Wp$ in $W$, which is one or two according to 
whether $\dim(V)$ is odd or even.)
\end{proposition}
\begin{proof} 
This is a direct consequence of the construction of the simple modules
for $C(V_0)$.  See, for example, the discussion around 
\cite[Lemma II.6.5]{BW}.
\end{proof}

\begin{proposition}
\label{p:wedgepm}
Given $w \in W$ let $\det(1 - w)$ denote the determinant of $Id_V - w$
acting on $V$.
Let $\chi_{\wedge^\pm}$ 
denote the character of $\wedge^\pm V$ (defined
in \eqref{e:wedgepm})
regarded now as a virtual representation of $W$.  Then
\[
\chi_{\wedge^\pm}(w) = {\det}(1 -w)
\]
for all $w$ in $W$.
\end{proposition}
\begin{proof} 
This is \cite[Lemma 2.1.1]{R}.
\end{proof}

\subsection{The affine Hecke algebra}
\label{ss:aha}
In the setting of Section \ref{ss:rs}, let $W_\ext=X\rtimes W$ denote
the extended Weyl group. We denote the elements of $W_\ext$ by $w
\mathsf a^x$, where $w\in W,$ $x\in X$ and $\mathsf a$ is a
symbol. The length function $\ell$ on $W_\ext$ is defined as in
\cite[1.4(a)]{L}:
\begin{equation*}
\ell(w\mathsf a^x)=\sum_{\al\in R^+,w(\al)\in
  R^-}|(x,\check\al)+1|+\sum_{\al\in R^+,w(\al)\in R^+}|(x,\check\al)|.
\end{equation*}
Let $\le$ be the order on $R^\vee$: $\check\al\le\check\beta$ if and
only if $\check\beta-\check\al\in\bZ_{\ge 0}\langle\check\al:
\al\in\Pi\rangle,$ and set $\Pi_m=\{\beta\in R: \check\beta\text{
  minimal for }\le\}.$

Let $Q = \bZ R$ denote the
integral span of $R$. Then $W_\aff := Q \rtimes W$ is a Coxeter group
generated by $S=\{s_\al: \al\in\Pi\}\cup\{s_\al\mathsf a^\al:\al\in
\Pi_m\}.$ Fix an indeterminate $q$ and a function $L:S\to \mathbb N$ such
that $L(s)=L(s')$ whenever $s$ and $s'$ are conjugate by $W_\ext.$ If
$\al\in\Pi$, let $S(\al)$ be the connected component of $\al$ in the
Coxeter graph $(W_\aff,S)$. If $\check\al\in 2Y$, then $S(\al)$ is an
affine diagram of type $C$, and therefore it has a unique nontrivial
automorphism. Let $\wti s_\al$ denote the image of $s_\al$ under this automorphism.

The affine Hecke algebra $\caH
= \caH(\Phi,\mathsf q,L)$ is the complex associative algebra with unit
over
$\bC[\mathsf q,\mathsf q^{-1}]$ with basis $\{N_x \; |
\; x\in W_\ext\}$ subject to the relations:
\[
N_xN_y = N_{xy} \qquad \text{if $\ell(xy) = \ell(x) + \ell(y)$}
\]
and
\[
(N_s -\mathsf q^{L(s)})(N_s+\mathsf q^{L(s)}) = 0 \qquad \text{for $s\in S_\aff$}.
\]
Set $T = \Hom_\mathsf{alg}(X,\bC^\times)$.
The center of $\caH$ is isomorphic to $\bC[T]^W$, the $W$
invariants in the coordinate ring of $T$.  
A version of Schur's Lemma implies that the center acts by scalars
in any irreducible $\CH$ module $\CX$, and hence we can attached a $W$ orbit
in $T$ to $\CX$ called the central character of $\CX$.  
If the orbit consists of real positive valued functions on $X$, then we say that 
$\CX$ has {real} central  character.
 
 In the equal parameter case, the Borel-Casselman equivalence gives
 natural notions of tempered (and discrete series) $\caH$ modules.  In
 the general case, a version of Casselman's weight criterion can be
 formulated to define tempered and discrete series $\caH$ modules.
Opdam has shown that these definitions have the expected analytic
interpretations. See \cite[Section 2]{O} for
a summary.

\subsection{The graded affine Hecke algebra}
\label{ss:daha}
In the setting of Section \ref{ss:rs}, fix a $W$-invariant map $c:R\to
\mathbb N$, and set $c_\alpha = c(\alpha)$.  Let $\mathsf r$ denote an indeterminate. Lusztig's graded affine
Hecke algebra \cite{L} $\bH=\bH(\Phi,\mathsf r,c)$ attached to the root datum
$\Phi$ and with parameter function $c$ is the complex associative
algebra over $\bC[\mathsf r]$ with unit generated by the symbols $\{t_w\; | \; w\in W\}$ and
$\{t_f \; | \; f \in S(V^\vee)$\}, subject to the relations:
\begin{enumerate}
\item[(1)]
The linear map from the group algebra $\bC[W] = \bigoplus_{w \in W} \bC w$ 
to $\bH$ taking $w$ to $t_w$ is an injective map of algebras.

\item[(2)]
The linear map from the symmetric algebra $S(V^\vee)$ to $\bH$ taking
an element $f$ to $t_f$ is an injective map of algebras.
\end{enumerate}
As usual, we  view $\bC[W]$
and $S(V^\vee)$ as subalgebras of $\bH$,
and write $f$ instead of $t_f$ in $\bH$.
The final relation is
\begin{enumerate}
\item[(3)]
\begin{equation*}\label{hecke}
\omega t_{s_\alpha}-t_{s_\alpha} s_\alpha(\omega)= c_\alpha \mathsf r
(\alpha,\omega),\quad \alpha\in \Pi,~ \omega\in V^\vee;
\end{equation*}
here $s_\alpha(\omega)$ is the element of $V^\vee$ obtained by $s_\alpha$ acting on $\omega$.
\end{enumerate}

The center $Z(\bH)$ of $\bH$ is 
$S(V)^W$.  Again a version of Schur's Lemma implies that the center
acts by scalars in any irreducible $\bH$ module $X$, and hence determines
a $W$ orbit in $V$ called the central character of $X$.  If
the $W$ orbit actually lies in $V_0$, then we say $X$ has real central
character. 

Finally, any $\bbH$ module $X$ can be restricted to $\bbC[W]$ to obtain a 
representation of $W$.  This descends to a map
\begin{equation}
\label{e:resdaha}
\res_W\; : \; \caR(\bH) \longrightarrow \caR(W).
\end{equation}

\subsection{Relation between $\caH$ modules and $\bH$ modules.}
\label{ss:relation}
In the setting of Sections \ref{ss:aha} and \ref{ss:daha}, specialize
from now on 
$\mathsf q\in \bR_{>1}$, $\mathsf r=\log q>0,$  and
let $c$ denote the parameter function for $\bH$ defined as follows:
\begin{equation*}
c(\al)=\left\{\begin{matrix} 2L(s_\al), &\al\in\Pi,\check\al\notin
    2Y,\\
L(s_\al)+L(\wti s_\al),&\al\in\Pi,\check\al\in 2Y;\end{matrix}\right.
\end{equation*}
here $s_\al,\wti s_\al, L$ are as in Section \ref{ss:aha}. 
The main results of \cite[Section 10.9]{L} establish an equivalence
between the category of $\caH$ modules $\C R^0(\C H)$ with real
 central character and the category of 
$\bH$ modules with real central character.  We simply collect the
properties of this equivalence we shall need below:
\begin{enumerate}
\item[(a)]
There is a bijection between tempered $\caH$ modules with real central
character and tempered $\bH$ modules with real central character.

\item[(b)] If $\CX$ and $X$ correspond under the equivalence of categories,
  then define 
(with notation as in \eqref{e:resdaha})
\begin{equation}
\label{e:wcorr}
\res_W(\CX) := \res_W(X).
\end{equation}
This extends to a linear map 
\begin{equation}\label{e:resaha}
\res_W:\C R^0(\C H)\to \C R(W).
\end{equation}
\end{enumerate}

\subsection{The Dirac operator}
\label{ss:dirac}
Let $\omega_1, \dots, \omega_r$ denote an orthonormal basis of
$V_o$ (with respect to $\langle~,~\rangle$).  Set 
\[
\wti \omega_i = \omega_i - \frac12 \sum_{\beta > 0} c_\beta (\omega,\beta^\vee)f_{s_\beta} \in \bbH.
\]
Following \cite{BCT}, define
\begin{equation}
\label{e:dirac}
\caD = \sum_i \wti \omega_i \otimes \omega_i \in \bbH \otimes C(V_0).
\end{equation}
Then it is easy to verify that $\caD$ is well-defined independent of
original choice of orthonormal basis.
Given an $\bbH$-module $X$ and a complex simple module $S$ for $C(V_0)$, let $D$ denote the image
of $\caD$
in endomorphisms of $X \otimes S$,
\[
D\in \End_\bbC(X\otimes S).
\]
Then $D$ is called the Dirac operator for $X$ (and $S$).  Define
the Dirac cohomology of $X$ to be
\[
H_D(X) = \ker(D)/(\ker(D) \cap \im(D)).
\]

\begin{proposition}
\label{p:weq}
Let $\rho$ denote the inclusion
\[
\bbC[\wti W] \lra \bbH \otimes C(V_0)
\]
obtained by linearly extending the map
\[
\wti w \mapsto t_{p(\wti w)} \otimes \wti w.
\]
Then
\[
\rho(\wti w) \caD = \sgn(p(\wti w)) \caD
\]
as elements of $\bbH \otimes C(V_0)$.  Thus left multiplication
by $\rho(\wti w)$ defines a representation of $\wti W$ on $H_D(X)$.
\end{proposition}
\begin{proof}
This is \cite[Lemma 3.4]{BCT}.
\end{proof}

Consider the Casimir element
\[
\Omega_{\wti W} 
= \frac14 \sum_{\stackrel{\alpha, \beta >0}{s_\alpha(\beta)>0}} c_\alpha c_\beta 
\frac{|\alpha^\vee|}{|\alpha|} \frac{|\beta^\vee|}{|\beta|} \alpha \beta \in C(V_0).
\]
In fact $\Omega_{\wti W}$ is an element of $\bC[\wti W]$, and it is
central in $\bC[\wti W]$; e.g.~\cite[Section 3.4]{BCT}.  
Given an irreducible representation $\tilde \sigma$ of $\wti W$, let 
\begin{equation}
\label{e:a-omega}
a(\tilde \sigma) = \text{the scalar by which $\Omega_{\wti W}$ acts in $\tilde \sigma$.}
\end{equation}

\begin{proposition}
\label{p:d-squared}
Suppose $X$ is an irreducible $\bH$ module with central character represented by $\nu \in V$.
If $\tilde \sigma$ is an irreducible representation of $\wti W$ such that
\[
\Hom_{\wti W}\left (\sigma, H_D(X) \right ) \neq 0,
\]
then
\[
\langle \nu,\nu \rangle = a(\tilde \sigma),
\]
with notation as in \eqref{e:a-omega}.
\end{proposition}

\begin{proof}
Theorem 3.5 of \cite{BCT} shows that
\begin{equation}
\label{e:d-squared}
\caD^2 = -\Omega_\bH \otimes 1 + \rho\left( \Omega_{\wti W} \right)
\end{equation}
where $\Omega_\bH$ is a central element of $\bH$ which acts by the squared
length of the central character in any irreducible $\bH$ module.  Hence
if $v\neq 0$ is in the $\tilde \sigma$ isotypic component of the kernel of $\caD$ acting
on $X \otimes S$, then applying both sides of \eqref{e:d-squared} to $v$ gives
\[
0 = -\langle \nu,\nu \rangle + a(\tilde \sigma).
\]
This proves the proposition.
\end{proof}

\subsection{The Dirac index}
\label{ss:index}
Retain the setting of Section \ref{ss:daha}, and
recall the irreducible $\wti \Wp$ modules $S^\pm$
introduced in Section \ref{ss:spinor}.  Define the Dirac index 
\[
I \; : \; \caR(\bH) \lra \caR(\wti \Wp)
\]
as
\begin{equation}
\label{e:index}
I(X) = \res_{\Wp}(X) \otimes (S^+ - S^-);
\end{equation}
here $\res_{\Wp}(X)$ denote the restriction of $X$ to $\Wp$ pulled back to 
an element of $\caR(\wti \Wp)$.
The remainder of this section will be
devoted to explaining the relationship between $I(X)$ and $H_D(X)$.

Suppose first that $\dim(V_0)$ is even, and let $(\gamma,S)$ denote the
unique complex simple $C(V_0)$ module up to equivalence.  For a fixed
$\bbH$ module $X$, let $D$ denote the Dirac operator defined in the
previous section.
As remarked
in Section \ref{ss:spinor}, the restriction of $S$ to $C(V_0)_\ev$
splits into simple modules as $S = S^+ \oplus S^-$.  Thus for $s \in S^\pm$
and $v \in V_0 \subset C(V_0)_\mathsf{odd}$, $\gamma(v)s \in S^\mp$.  Hence
$D$ maps $X \otimes S^+$ to $X
\otimes S^-$, and $X \otimes S^-$ to $X
\otimes S^+$.  
We let $D^+$ and $D^-$ denote the respective restrictions,
\[
D^\pm \; : \; X\otimes S^\pm \lra X\otimes S^\mp,
\]
and set
\begin{equation}
\label{e:hdpm}
H_D^\pm(X) = \ker(D^\pm)/(\ker(D^\pm) \cap \im(D^\mp).
\end{equation}
According to Proposition \ref{p:weq},
\begin{equation}
\label{e:hdpm1}
H_D^\pm(X) \in \caR(\wti \Wp).
\end{equation}

Next suppose $\dim(V_0)$ is odd, and recall the two simple
$C(V_0)$ module structures  $S^+$ and $S^-$ (on the same complex
vector space).  Fix an $\bbH$ module $X$ and define
\[
D \; : \; X \otimes S^+ \lra X\otimes S^+.
\]
We can compose this with the vector space identity map $S^+ \to S^-$
to obtain
\[
D^+ \; : \; X \otimes S^+ \lra X\otimes S^-.
\]
Reversing the roles of $S^+$ and $S-$ we obtain
\[
D^- \; : \; X \otimes S^- \lra X\otimes S^+.
\]
We then define $H^\pm_D(X)$ via \eqref{e:hdpm}.  According
to \eqref{e:signtwist} and Proposition \ref{p:weq}, we once again have
\begin{equation}
\label{e:hdpm2}
H_D^\pm(X) \in \caR(\wti \Wp).
\end{equation}

\begin{proposition}
\label{p:index}
Fix an $\bbH$ module $X$, and recall the notation of
\eqref{e:index}, \eqref{e:hdpm1}, and \eqref{e:hdpm2}.  Then
\begin{equation}
\label{e:indexformula}
I(X) = H_D^+(X) - H^-_D(X) 
\end{equation}
as elements of $\caR(\wti \Wp)$.
\end{proposition}

\begin{proof}
Equation \eqref{e:d-squared} implies that $D^\pm \circ D^\mp$ are diagonalizable
linear operators.  The proposition then reduces to simple linear algebra.
\end{proof}

Proposition \ref{p:index} explains why $I(X)$ is called the Dirac
index of $X$. Results like Proposition \ref{p:d-squared} place rather strict 
limitations
on the possible structure of $I(X)$, and hence give
nontrivial information about the structure of $X$ as a $\Wp$ module.
In Section \ref{ss:equal}, we use this idea (together with the main results
of \cite{C}) to explicitly identify the numerator in
Theorem \ref{t:intro}.

\section{main results}
\label{s:main}

\subsection{Elliptic representations of $W$}
\label{ss:ell}
An element $w \in W$ is called elliptic if the action of $w$ on $V$
has no fixed points; equivalently (in the notation of Proposition
\ref{p:wedgepm}) if
\[
{\det}(1-w) \neq 0.
\]
The set of elliptic element of $W$ will be denoted $W_\elli$.  
We can make the same definition  for the subgroup $\Wp$ defined
in \eqref{e:wp}.  A quick check (in the even case) shows that
$\Wp_\elli = W_\elli$.

Following \cite[(2.2.1)]{R}, we define a bilinear pairing
on $\caR(W)$ defined on irreducible representations $\sigma_1$ and
$\sigma_2$ via
\[
e_W(\sigma_1,\sigma_2) = 
\sum_i (-1)^i 
\dim \Hom_W\left (\sigma_1\otimes \wedge^{i}V, \sigma_2 \right ).
\]
Then \cite[(2.2.2)]{R} shows that taking characters induces
an isomorphism
\begin{equation}
\label{e:ellisom}
\ol{\caR}(W) := \caR(W)/\ker(e_W)\;  \lra \; \bbC[W_\elli]^W
\end{equation}
onto the class functions on $W$ vanishing off $W_\elli$.  
(More intrinsically, the kernel of $e_W$ consists of the span of
 representations induced from proper parabolic subgroups of $W$.)
Continue to write
$e_W$ for the induced nondegenerate form on $\ol{\caR}(W)$.

\begin{remark-plain}
\label{r:indexvanishing}
By taking characters of both sides of \eqref{e:index} and using \eqref{e:spm-intro} we see
that $I(X) = 0$ if and only if the character of $\res_\Wp(X)$ vanishes on $\Wp_\elli$.
By the parenthetic remark after \eqref{e:ellisom}, this implies that $I(X)$ vanishes
if and only if $X$ is in the span of $\bH$ modules which are induced from
proper parabolic subalgebras of $\bH$.
\end{remark-plain}

In the setting of Section \ref{ss:cliff} and \ref{ss:repring}, 
let $\caR_g(\wti \Wp)$
denote the complex linear combinations of genuine irreducible
representations of $\wti \Wp$.  (As usual, a representation of $\wti
\Wp$ is called genuine is it does not factor to $\Wp$.)  
Recall the (genuine)
$\wti \Wp$ modules $S^\pm$ defined in Section \ref{ss:spinor}.
Consider the map
\begin{equation}
\label{e:tau}
\sfi \: \; : \caR(W) \lra \caR_g(\wti \Wp)
\end{equation}
defined by 
\begin{equation}
\label{e:tau2}
\sigma \mapsto  \res_{\Wp}(\sigma) \otimes (S^+ - S^-);
\end{equation}
where $\res_{\Wp}(\sigma)$ denotes the restriction of $\sigma$ to $\Wp$ pulled back to $\wti \Wp$;
c.f.~\eqref{e:index}.
Proposition \ref{p:s-squared} shows that $\sfi$ vanishes on any
element of $\caR(W)$ whose character has support in the complement
of $W_\elli = \Wp_\elli$.  
Thus \eqref{e:ellisom}
implies that $\sfi$ descends
to an injection
\begin{equation}
\label{e:taubar}
\sfi \: \; : \ol{\caR}(W) \lra \caR_g(\wti \Wp).
\end{equation}

\begin{proposition}
\label{p:isometry}
The map $\sfi$ defined in \eqref{e:taubar} 
satisfies
\[
\left (\sfi(\sigma_1),
\sfi(\sigma_2)\right)_{\wti \Wp}
=2{e}_W(\sigma_1,\sigma_2)
\]
for all $\sigma_i \in \ol{\caR}(\Wp)$.
\end{proposition}
\begin{proof}
For irreducible representations $\sigma_1$ and $\sigma_2$ of $\Wp$, we compute
\begin{align*}
\left ( \sfi(\sigma_1), \sfi(\sigma_2) \right)_{\wti \Wp} &=
\left ( \sigma_1\otimes(S^+ - S^-), \sigma_2\otimes(S^+-S^-) \right)_{\wti \Wp} 
\\ &=
\left ( \sigma_1\otimes(S^+ - S^-)^{\otimes 2}, \sigma_2 \right)_{\wti \Wp} 
\\ &=
\frac{2}{[W:\Wp]}\left ( \sigma_1\otimes 
\wedge^\pm V, \sigma_2 \right)_{\wti \Wp}
\\ &=
\frac{2}{[W:\Wp]}\left ( \sigma_1\otimes 
\wedge^\pm V, \sigma_2 \right)_{\Wp}
\end{align*}
for the second equality, we have used that $S^\pm$ are self-dual; for
the third, we have used Proposition \ref{p:s-squared}; for the fourth,
we have used that the representations being paired all factor to $\Wp$.
On the other hand, since the character $\chi_{\wedge^\pm(V)}$
is supported on $W_\elli$ (by Proposition \ref{p:wedgepm})
and since $\Wp_\elli = W_\elli$,
\begin{align*}
\frac{2}{[W:\Wp]} 
\left ( \sigma_1\otimes 
\wedge^\pm V, \sigma_2 \right)_{\Wp}
&=
\frac{2}{[W:\Wp]} \frac{1}{|\Wp|}\sum_{x\in \Wp} 
\ol{\chi_{\sigma_1 \otimes \wedge^\pm(V)}(x)} \chi_{\sigma_2}(x)
\\ &=
\frac{2}{|W|}\sum_{x\in W} \ol{\chi_{\sigma_1 \otimes \wedge^\pm(V)}(x)} \chi_{\sigma_2}(x)
\\ &= 
2\sum_i (-1)^i 
\dim \Hom_W\left (\sigma_1\otimes \wedge^{i}V, \sigma_2 \right ) 
\\ &=
2e_W(\sigma_1,\sigma_2).
\end{align*}
This completes the proof.
\end{proof}

\subsection{Relation with the Euler-Poincar\'e pairing; proof of Theorem \ref{t:intro2}}
\label{ss:ss}
In the setting of Section \ref{ss:aha}, let $\CX$ and $\CY$ be two
irreducible $\caH$ modules.  Following \cite{SS} define
\[
\EP(\CX,\CY) = \sum_{i\geq 0} (-1)^i \dim \mathrm{Ext}^i_\caH(\CX,\CY).
\]

\begin{theorem}[Schneider-Stuhler, Opdam-Solleveld]
\label{t:ss}
Suppose $\CX$ is an irreducible discrete series $\CH$ module and $\CY$ is
an irreducible tempered $\CH$ module.
Then $\EP(\CX,\CY) = 0$
unless $\CX \simeq \CY$ in which case
\[
\EP(\CX,\CX) = 1.
\]
\end{theorem}
\begin{proof}
This is \cite[Theorem 3.8]{OS}.
\end{proof}

The connection with the previous section is as follows.
Let $\caR^0_\temp(\caH)$ denote the subspace of $\caR(\caH)$
generated by irreducible tempered modules with real central character.
Set
\begin{equation}
\label{e:RHbar}
\ol\caR^0_\temp(\caH) = \caR^0_\temp(\caH)/\ker(\EP)
\end{equation}
and continue to write $\EP$ for the induced nondegenerate form.  

\begin{theorem}[Reeder, Opdam-Solleveld]
\label{t:reeder} 
Recall the map $\res_W$ of \eqref{e:resaha} and the notation of
\eqref{e:ellisom} and \eqref{e:RHbar}.  Then $\res_W$ restricts to a
linear map
\[
\res_W\; : \; \ol \caR^0_\temp(\caH) \lra \ol\caR(W)
\]
which satisfies
\begin{equation}
\label{e:resisom}
{e}_W\left (\res_W(\CX),  \res_W(\CY) \right) = {EP}(\CX,\CY)
\end{equation}
for all $\CX, \CY \in \ol\caR(\caH)$.
\end{theorem}
\begin{proof}
For equal parameters, this is \cite[Theorem 5.10.1]{R}.   The general case follows from
Proposition 3.9(1) and Theorem 3.2(c) in \cite{OS}.
\end{proof}

\begin{lemma}
\label{l:relation}
In the setting of Section \ref{ss:daha}, let $X$ and $Y$ be two
irreducible tempered $\bH$ modules with real central character,
and write $\CX$ and $\CY$ for the corresponding tempered
$\CH$ modules (Section \ref{ss:relation}).
Recall the Dirac index $I(X) = X\otimes (S^+ - S^-)$ of \eqref{e:index}.
Then
\[
\left (I(X),I(Y)\right)_{\wti W}=2 EP(\CX,\CY).
\]
If we further assume that $X$ is a discrete series module, then
\begin{equation}
\label{e:index2}
\left (I(X),I(X)\right)_{\wti W}=2.
\end{equation}
\end{lemma}
\begin{proof}
From \eqref{e:index} and \eqref{e:tau},
it follows that $I(X) = \sfi\left (\res_W(X)\right )$.
Proposition \ref{p:isometry} and Theorem \ref{t:reeder}
imply that
\[
\left( I(X), I(Y)\right )_{\wti W'} = 2\EP(\CX,\CY).
\]
This is the first assertion of the lemma, and \eqref{e:index2} then follows from Theorem \ref{t:ss}.
\end{proof}

\begin{theorem}\label{t:diff}
In the setting of Section \ref{ss:daha}, suppose $X$ is an irreducible discrete series module for $\bH$.
Then there exist inequivalent genuine irreducible $\wti \Wp$
representations $\tilde \sigma^+$ and $\tilde \sigma^-$ such that
\[
I(X) = \tilde \sigma^+ - \tilde \sigma^-.
\]
If $\dim(V_0)$ is odd (so that $\wti \Wp=\wti W$ by definition), then
$\tilde\sigma^- := \tilde \sigma^+ \otimes \sgn$.
If $\dim(V_0)$ is even (so that $\wti \Wp$ is an index two normal subgroup of $\wti W$),
the action of
$H = \wti W\bigr /\wti \Wp$ interchanges $\tilde \sigma^+$ and $\tilde \sigma^-$; equivalently,
there exists an irreducible representation $\tilde \sigma$ of $\wti W$ which splits as 
$\tilde \sigma^+ \oplus \tilde \sigma^-$ when restricted to $\wti \Wp$.
\end{theorem}

\begin{proof}
By \eqref{e:index2}, there exist
inequivalent irreducible $\wti \Wp$ representations $\tilde \sigma_1$
and $\tilde \sigma_2$ and coefficients $a_i \in \{\pm 1\}$ such that
\[
I(X) = a_1 \sigma_1 + a_2 \sigma_2.
\]
If $\dim(V_0)$ is odd, then $S^\pm \otimes \sgn = S^\mp$, and hence
$I(X) \otimes \sgn = -I(X)$.  This implies $\sigma_2 = \sigma_1 \otimes \sgn$
and $a_1 = -a_2$, as claimed.  On the other hand, if $\dim(V_0)$ is even, then the action of
$H = \wti W\bigr /\wti \Wp$ interchanges $S^+$ and $S^-$.  Hence $H$ acts
by $-1$ on $I(X)$, and so $H$ interchanges $\sigma_1$ and $\sigma_2$
and $a_1 = -a_2$.  
\end{proof}

\begin{proof}[{Proof of Theorem \ref{t:intro2}}]
By definition, $I(X) = \res_{\Wp}(X) \otimes (S^+ - S^-)$.  So applying Theorem
\ref{t:diff}, taking characters, and dividing by $\chi_{S^+} -
\chi_{S^-}$ gives the conclusion of Theorem \ref{t:intro2}.
\end{proof}

\subsection{The case of equal parameters and the proof of Theorem \ref{t:intro}}
\label{ss:equal}
Let $\frg$ denote the complex semisimple Lie algebra constructed from the root system
$\Phi$.  In particular this construction fixes a Cartan subalgebra $\frh$ of $\frg$ isomorphic to $V$,
and hence $\frh$ is equipped with the bilinear form $\langle~,~\rangle$.
Let $e$ be a nilpotent element
of $\frg$, and consider an $\mathfrak{s}\mathfrak{l}_2$-triple $\{e,f,s\}$ with $s\in \frh$.  Then
\begin{equation}
\label{e:he}
h(e) := \langle s,s \rangle
\end{equation}
is well-defined independent of the choice of triple.
Let $\Ahat_0(e)$ denote the irreducible representations of the component group
of the centralizer of $e$ in $\mathrm{Ad}(\frg)$ which appear in the Springer correspondence.

According to the classification of \cite{KL} and \cite{L}, the pair
$(e,\phi)$ parametrizes an irreducible tempered $\bH$ module
$X(e,\phi)$ with (real) central character represented by $\mathsf r c s/2$,
where $s$ is the semisimple element mentioned above, such that the
character of $\res_W(X(e,\phi))$ is $\chi_{e,\phi}$. 

\begin{theorem}
\label{t:diffeq}
Suppose that the parameter function $c$ in the definition of $\bH$ is
constant.
Let $X$ be an irreducible tempered $\bH$ module with real central character
such that $I(X)$ is nonzero (cf.~Remark \ref{r:indexvanishing}).  Then 
\begin{equation}
\label{e:diffeq}
\left (I(X),I(X)\right)_{\wti W}=2,
\end{equation}
and the conclusion of Theorem \ref{t:diff} holds.
\end{theorem}

\begin{proof} 
Let $\C X$ be the irreducible tempered module for $\C H$ with real central character
corresponding to $X$. 
The proof of Theorem \ref{t:diff} shows that the theorem amounts to establishing that
$\EP(\C X,\C X) =1$. We have already seen that this is the case when
$\C X$ is a discrete series, so we assume that $\C X$ is not a
discrete series. 
The assumption that $I(X) \neq 0$
implies that $\C X$ has nonzero image in $\ol \caR^0_\temp(\bH)$,
and hence the image of $\res_W(X)$ is nonzero.
By the result of Reeder mentioned in the introduction, this implies
that $X=X(e,\phi)$, for some quasidistinguished, not distinguished
nilpotent element $e$, and $\phi\in \Ahat_0(e).$ Fix such an element
$e$.  Let $\C R_e(W)$
denote the subspace of $\C R(W)$ spanned by
$\{\res_W(X(e,\phi)):\phi\in\Ahat_0(e)\}$, and let $\C R_0(A(e))$
denote the subspace of $\C R(A(e))$ spanned by $\Ahat_0(e)$. By
\cite[(3.4.1),(3.4.3)]{R}, the map
$\C R_0(A(e))\to \C R_e(W)$, $\phi\mapsto
\res_W(X(e,\phi))$ induces a linear isomorphism $\ol \caR_0(A(e))\to
\ol\caR_e(W)$ which is an isometry with respect to the elliptic
pairings on the two spaces. The definition of $\ol\caR_0(A(e))$ is as in
\cite[Section 3.2]{R}. Since the elliptic pairing of the trivial
representation with itself is one, notice that this implies the
conclusion for $X=X(e,\triv).$

To finish, we need an empirical observation, namely that if $e$ is quasidistinguished,
but not distinguished,
the space $\ol\caR_0(A(e))$ is one dimensional. This amounts to a
case-by-case verification that $X(e,\phi)\pm X(e,\triv)$ is a
combination of induced $\bH$-modules, for every $\phi\in \Ahat_0(e)$. In particular,
$\res_W(X(e,\phi))$ and $\res_W(X(e,\triv))$ have the same image, up to a
sign, in $\ol\caR_e(W)$. The claim follows.
\end{proof}

Arguing as in the proof of Theorem \ref{t:intro2} at the end of the previous section, we obtain the following.
\begin{corollary}
\label{c:diffeq}
Suppose that the parameter function $c$ in the definition of $\bH$ is constant.
Then
Theorem \ref{t:intro2} holds for all irreducible
tempered $\bH$ modules $X$ with real central character 
such that the character of $\res_W(X)$ does not identically vanish on the set of elliptic elements 
of $W$.
\end{corollary}

Since
the character of $\res_W(X(e,\phi))$ is $\chi_{e,\phi}$,
 \eqref{e:intro} now follows from Corollary \ref{c:diffeq}. 
To complete the proof of Theorem \ref{t:intro}, 
it remains to identify the numerator in \eqref{e:intro} explicitly.

Recall from Section \ref{s:intro} the notation $\sigma(e,\phi)$ for the irreducible Springer
representation in the top degree homology.
\begin{lemma}
\label{l:diraclwt} Suppose
$\tilde \sigma$ is an irreducible representation of $\wti \Wp$ such that
\[
\left (\tilde \sigma, X(e,\phi) \otimes (S^+ - S^-) \right )_{\wti \Wp} \neq 0.
\]
Then
\[
\left (\tilde \sigma, \sigma(e,\phi) \otimes (S^+ - S^-) \right )_{\wti \Wp} \neq 0.
\]
\end{lemma}

\begin{proof}
For nilpotent elements $e$ and $e'$ in $\frg$, write $e' > e$ if the closure
of the nilpotent orbit through $e'$ contains $e$.
By the parametrization introduced above, this induces an order on the set of
irreducible tempered representations with real central character and the set of
irreducible representations of $W$.  
With these orders in place, consider the matrix
whose rows are indexed by irreducible tempered modules with real infinitesimal character,
whose columns are indexed by irreducible representations of $W$, and whose content measures
the restriction of a tempered  module to $W$.  The construction of \cite{KL} show
that this matrix is upper triangular with 1's on the diagonal; see the exposition of \cite[Section 3]{c:unip},
for example.  Thus in $\caR(W)$ we may write
\[
\res_W(X(e,\phi))=\sigma(e,\phi)+\sum_{e'>e} b_{e,e'} \res_W(X(e',\phi')),
\]
for integers $b_{e,e'}$.
Thus
\begin{equation}
\label{e:Iformula}
I(X(e,\phi))=\sigma(e,\phi)\otimes (S^+-S^-)+\sum_{e'>e} b_{e,e'} I(X(e',\phi')).
\end{equation}
By Propositions \ref{p:d-squared} and \ref{e:indexformula} applied to $I(X(e,\phi))$, the only irreducible
$\wti \Wp$ representations appearing on the right-hand side of \eqref{e:Iformula}
arise as restrictions of irreducible $\wti W$ representations $\tilde \sigma$ for
which $a(\tilde \sigma) = h(e)$.
By Propositions \ref{p:d-squared} and \ref{e:indexformula} applied to $I(X(e',\phi')$, the 
final sum on the right-hand side of \eqref{e:Iformula} can only contribute 
$\wti \Wp$ modules arising as the irreducible $\wti W$ representations $\tilde \sigma$ for
which $a(\tilde \sigma) = h(e')$ for $e' > e$.
But it is an empirical fact
that $h(e') > h(e)$ whenever
$e'>e.$ The proof is complete.
\end{proof}

\begin{lemma}
\label{l:c}
Assume that $e$ is a quasidistinguished nilpotent element of $\frg$, and recall
the notation of \eqref{e:wp}, \eqref{e:a-omega}, and \eqref{e:he}.   
\begin{enumerate}
\item[(a)]
If $\dim(V_0)$
is odd, there are exactly two inequivalent irreducible $\wti W$ representations $\tilde \sigma^+$
and $\tilde \sigma^- = \tilde \sigma^+ \otimes \sgn$ such that
\begin{equation}
\label{e:codd}
\left (\tilde \sigma^\pm, S^+- S^-\right)_{\wti W} \neq 0
\text{ and }
h(e) = a(\tilde \sigma^\pm).
\end{equation}

\item[(b)]
If $\dim(V_0)$ is even, there is a unique irreducible $\wti W$ representations $\tilde \sigma \simeq 
\tilde \sigma \otimes \sgn$ such that 
\begin{equation}
\label{e:ceven}
\left (\tilde \sigma, S^+ -S^-\right)_{\wti W} \neq 0
\text{ and }
h(e) = a(\tilde \sigma).
\end{equation}
When restricted $\wti \Wp$, $\tilde \sigma$ splits into two inequivalent irreducible
representations $\tilde \sigma^+$ and $\tilde \sigma^-$.
\end{enumerate}
\end{lemma}
\begin{proof}
Theorem 1.0.1(c) in \cite{C} implies that in either case there exists an irreducible
representation $\tilde \sigma$ of $\wti W$ satisfying \eqref{e:ceven} and, moreover,
$\tilde \sigma$ is unique up to possibly tensoring with $\sgn$.  A short argument
(using Proposition \ref{p:s-squared} and the fact that
$W_\elli \subset W_\ev$ if and only if $\dim(V_0)$ is even) leads
to the two cases indicated in the lemma.
\end{proof}

\begin{theorem}
\label{t:maineq} 
Suppose the parameter function $c$ in the definition of $\bH$ is
constant.  
Let $X=X(e,\phi)$ be an irreducible
tempered $\bH$ module such that the character $\chi_X$ of $\res_W(X)$ does not vanish identically
on $W_\ell$.  (According to \cite{R}, this is equivalent to requiring $e$ to be quasidistinguished.) 
Write  $\chi_{\tilde \sigma^+}$ and $\chi_{\tilde \sigma^-}$ for the characters of the irreducible
$\wti \Wp$ representations given in Lemma \ref{l:c}.  Fix $w \in \Wp_\elli = W_\elli$ and let
$\tilde w$ denote an element of $\wti W$ such that $p(\tilde w)= w$. 
Then, up to sign, 
\begin{equation}
\label{e:maineq}
\chi_X(w) = \frac{\chi_{\tilde \sigma^+}(\tilde w) - \chi_{\tilde \sigma^-}(\tilde w)}{\chi_{S^+}(\tilde w) - \chi_{S^-}(\tilde w)}.
\end{equation}
\end{theorem}

\begin{proof}
Theorem \ref{t:diff}, Lemma \ref{l:diraclwt}, and Lemma \ref{l:c} imply that
\[
\res_{W'}(X(e,\phi)) \otimes (S^+ - S^-)  = \tilde \sigma^+ - \tilde \sigma^-
\]
as $\wti \Wp$ representations (for some choice of labeling of $\sigma^\pm$ and $S^\pm$).
Taking characters gives the conclusion of the theorem.
\end{proof}

Since $\chi_{e,\phi} = \chi_{X(e,\phi)}$, Theorem \ref{t:maineq} is equivalent to Theorem \ref{t:intro}.

\ifx\undefined\bysame
\newcommand{\bysame}{\leavevmode\hbox to3em{\hrulefill}\,}
\fi

\end{document}